\documentclass[reqno, 11pt]{amsart}

\makeatletter
\@namedef{subjclassname@2020}{%
  \textup{2020} Mathematics Subject Classification}
\makeatother

\usepackage{ulem}

\usepackage{amsfonts, amsthm, amsmath, amssymb}
\usepackage{hyperref}
\hypersetup{colorlinks=false}

\usepackage{bbm}  

 \usepackage{tabu}

\usepackage[margin=1in]{geometry}

\usepackage{helvet}

\RequirePackage{mathrsfs} \let\mathcal\mathscr

\usepackage{hyperref}
\usepackage{dsfont}
\usepackage{upgreek}
\usepackage{mathabx,yfonts}
\usepackage{enumerate}

\numberwithin{equation}{section}

\newtheorem{theorem}{Theorem}[section] 
\newtheorem{lemma}[theorem]{Lemma}
\newtheorem{proposition}[theorem]{Proposition}
\newtheorem{corollary}[theorem]{Corollary}

\theoremstyle{definition}

\newtheorem{remark}[theorem]{Remark}

 \newtheorem*{notation}{Notation}

\renewcommand{\emph}[1]{\textit{#1}}

\renewcommand{\phi}{\varphi}

\renewcommand{\leq}{\leqslant}

\renewcommand{\geq}{\geqslant}

 \renewcommand{\b}{\mathbf{b}}

\DeclareSymbolFont{bbold}{U}{bbold}{m}{n}
\DeclareSymbolFontAlphabet{\mathbbold}{bbold}

\newcommand{\md}[1]{  \left(\textnormal{mod}\ #1\right)}
 
\renewcommand{\P}{\mathbb{P}}

\newcommand{\Z}{\mathbb{Z}}

\renewcommand{\b}{\mathbf}
\renewcommand{\c}{\mathcal}
\renewcommand{\epsilon}{\varepsilon}

\renewcommand{\leq}{\leqslant}
\renewcommand{\geq}{\geqslant}
\renewcommand{\#}{\sharp}

\DeclareMathOperator*{\Osum}{\sum{}^*}

\title
[The Bateman--Horn conjecture on average for generalized von Mangoldt Functions]
{The Bateman--Horn conjecture on average for generalized von Mangoldt functions}

\author{E. Sofos} 
\address{Dipartimento di Matematica\\
Universit{\`a} di Roma
\\Tor Vergata\\
 Via del  la Ricerca Scientifica
 \\ 00133, Rome, Italy}
\email{efthymios.sofos@uniroma2.eu}

\date{}

\begin{document} 
\begin{abstract} We study 
the 
Bateman--Horn conjecture for generalised von Mangoldt functions. 
As an application, we prove that for $k \in \{2, 3\}$, almost all Bouniakowsky polynomials 
represent integers that are a product of exactly $k$ primes.
\end{abstract}

\maketitle

\setcounter{tocdepth}{1}

\tableofcontents

\section{Introduction}   \label{s:intro}  
Let $k$ be a positive integer.
A positive integer is called 
$E_k$ when it has  
exactly $k$ distinct prime divisors.
A result of Chen \cite{Chen} states that 
the polynomial $x(x+2)$ represents infinitely many
integers that are $E_2$ or $E_3$. Due to the parity problem
it is not known which of the two cases occurs infinitely.
Iwaniec \cite{MR485740} later proved 
that $x^2+1$ takes on infinitely many values that are a 
prime or $E_2$; 
again it is not known which case occurs infinitely often.  
 The only  
irreducible integer polynomials 
that are known to represent 
infinitely many $E_k$  are linear.

A polynomial $P\in \mathbb{Z}[t]$ is  called \textit{Bouniakowsky} when  it is irreducible,
its leading coefficient is positive and for every prime $\ell$ there is $m\in \mathbb Z$ such that 
$\ell\nmid P(m)$. Polynomials that represent infinitely  many primes  necessarily satisfy all these conditions
and Bouniakowski  conjectured  that these are sufficient~\cite[page 328]{boun}. 
The analogous equivalent condition   for representing infinitely many $E_2$ numbers is not obvious. For example,
$P(t)=2t$ is not Bouniakowsky but still represents infinitely many such integers. However the density  
of $E_2$ among the values of $P(t)=2t$ is low and it is therefore more relevant to ask for a generalisation of the Bateman--Horn 
conjecture \cite{MR148632}. The conjecture states that for irreducible $P\in \Z[t]$ one has $$
\sum_{\substack{1\leq m\leq x\\P(m) \textrm{ prime } } }
\log P(m)
\sim \mathfrak S_{ P}x, $$ where   $ \mathfrak S_{P}$ is the singular series defined by 
$$  \mathfrak S_{P}:=\prod_{\substack{ \ell \ \mathrm{ prime}\\ \ell=2 }}^\infty 
\frac{1-\ell^{-1} \#\{s\in \mathbb F_\ell:P(s)=0\}}{1-\ell^{-1}}.$$ It seems that the right $E_2$-generalisation is as follows:  
an $E_2$ integer $m$ has the form $m=p_1^{\alpha} p_2^{\beta}$ for primes $p_1 \neq p_2$ and positive integers $\alpha,\beta$. 
We then define   $\mathcal L_2(m)= (\log p_1)(\log p_2)$.  For an irreducible $P\in \mathbb Z[t]$ we ought to have 
$$ \sum_{\substack{1\leq m\leq x\\P(m) \textrm{ is } E_2 } }
\mathcal L_2(P(m))
\sim \frac{\mathfrak S_{ P}}{2} x (\log |P|x), $$ where $
|P|$ is the  maximum of absolute values of the 
coefficients of $P$.
Our first result proves this  for $100\%$ of polynomials ordered by height. We use the notation
$P>0$ to denote    that the leading coefficient of $P$ is strictly positive. We denote the version of $ \mathfrak S_{P}$ in which one takes the product over all primes $\ell \leq \log x $ by $ \mathfrak S_{P}(x)$.
\begin{theorem}
\label{thm:bounia2}  Fix any $d\in \mathbb N$ and $\delta>1$. For any $x,H\geq 1$ with 
$ (\log H)(\log \log H)<x \leq (\log H)^\delta$ we have  $$ \sum_{\substack{ P\in \Z[t], \deg(P)=d \\ P>0, |P|\leq H}}
\left|\sum_{\substack{1\leq m\leq x\\P(m) \ \mathrm{ is } \ E_2 } }\mathcal L_2(P(m))
-\frac{\mathfrak S_{ P}(x)}{2} x \log H\right|^2 \ll H^{d+1} \frac{(x\log H)^2}{\log x},$$ where the implied constant depends only on $\delta$ and $d$. In particular, for $100\%$ of degree $d$ integer polynomials $|P|\leq H$ with a positive leading coefficient 
one has   $$\sum_{\substack{1\leq m\leq x\\P(m) \ \mathrm{ is }\  E_2 } }\mathcal L(P(m))
=\frac{\mathfrak S_{ P}(x)}{2} x \log H+O\Bigg(
 \frac{x\log H}{(\log x)^{1/4}} \Bigg)$$ uniformly for  $x\leq (\log H)^\delta$. Furthermore, $100\%$ of degree $d$ Bouniakowsky polynomials $P$ represent more than   $(\log |P|)^\delta$ integers that are $E_2$. \end{theorem} 
This generalises the analogous statement 
for primes in \cite[Theorem 1.9]
{MR4542704}. The proofs depend on expressing $\mathcal L_2$ through a linear combination of 
  generalised von Mangoldt functions, which are defined by   
$$\Lambda_k(m):= 
\sum_{\substack{ c\in \mathbb N \\ c\mid m }} \mu(c) \Big(\log \frac{m}{c} \Big)^k, \ \ \ k,m \in \mathbb N.$$ 
These were initially investigated by Selberg for $k=2$ in his proof of the Prime Number Theorem, and subsequently studied for
$k \geq 3$ by Bombieri \cite{MR491570, MR396435} and Friedlander–Iwaniec \cite{MR519891,MR1399351}, among others. 
The generalisation of 
Bateman--Horn's conjecture to $\Lambda_k$
seems to be 
$$  
\sum_{\substack{1\leq m\leq x\\P_i(m)>0 \ \forall i} } \Lambda_{k_1}(P_1(m))\cdots \Lambda_{k_n}(P_n(m)) \sim 
\mathfrak S_{\b P} x 
\prod_{i=1}^n  k_i 
(\log |P_i|x)^{k_i-1}  
,$$  where $P_i$ are pairwise coprime
irreducible integer polynomials and 
$$ \mathfrak S_{\b P} =
\prod_{\substack{ \ell \ \mathrm{ prime}\\ \ell=2 }}^\infty 
\frac{1-\ell^{-1} \#\{s\in \mathbb F_\ell:P_1(s)\cdots P_n(s)=0\}}
{(1-\ell^{-1})^n}.$$
 We prove this in $100\%$ of the cases. Denote by $ \mathfrak S_{\b P}(x)$ the truncated version of $ \mathfrak S_{\b P}$ in which the product is over primes $\ell \leq \log x $.
Fix $n,d_1,\ldots, d_n \in \mathbb N$ and let  $$ \textrm{Poly}(H):=\{\b P\in (\Z[t])^n: 
P_i>0,\deg(P_i)=d_i, |P_i|\leq H
\ \textrm{ for } i=1,\ldots, n\}.$$ 
\newline
The following is the main 
result in this paper:
\begin{theorem}  \label{thm:Lambda_k}   Fix any $n\in \mathbb N, \b d, \b k \in \mathbb N^n,\delta>n$ and let 
$d=\sum_{i=1}^n d_i, k=\sum_{i=1}^n k_i$.  Then for any $x,H\geq 1$ with  $ (\log H)^n<x \leq (\log H)^\delta$ we have 
$$ \sum_{  \mathbf P\in  \mathrm{Poly}(H)}  \left| \sum_{\substack{1\leq m\leq x\\P_i(m)>0 \ \forall i} } 
\hspace{-0.3cm}\Lambda_{k_1}(P_1(m))\cdots \Lambda_{k_n}(P_n(m))  - \mathfrak S_{\b P}(x) x (\log H)^{k-n}
\prod_{i=1}^n  k_i  \right|^2  \! \!  \!  \ll  \! \!   \frac{H^{d+n} (x (\log H)^{k-n})^2}
 {\min\{ \log x,x(\log H)^{-n}\}},$$ where the implied constant depends only on $\delta,\b d,\b k$ and $n$.    \end{theorem} 
Thus, for almost all $P_1,\ldots, P_n$ one has $$ \sum_{\substack{1\leq m\leq x\\P_i(m)>0 \ \forall i} } 
\hspace{-0.3cm}\Lambda_{k_1}(P_1(m))\cdots \Lambda_{k_n}(P_n(m))  \sim (k_1\cdots k_n)  \mathfrak S_{\b P}(x) x (\log H)^{k-n}.$$
as soon as $x/(\log H)^n\to \infty$.

Taking suitable combinations of $\Lambda_k$ 
allows us to also handle  $E_3$ numbers. For a strictly positive  $E_3$ integer $m$
let  $\mathcal L_3(m)$  be the product of the three distinct prime divisors of $m$. \begin{theorem}
\label{thm:bounia33}  Fix  $d\in \mathbb N$ and $\delta>1$. For any $x,H\geq 1$ with 
$ (\log H)(\log \log H)<x \leq (\log H)^\delta$ we have  $$ \sum_{\substack{ P\in \Z[t], \deg(P)=d \\ P>0, |P|\leq H}}
\left|\sum_{\substack{1\leq m\leq x\\P(m) \ \mathrm{ is } \  E_3 } }\mathcal L_3(P(m))
-\frac{\mathfrak S_{ P}}{6} x (\log H)^2\right|^2 \ll  H^{d+1} \frac{(x\log H)^4}{\log x},$$ where the
implied constant depends only on $\delta$ and $d$. In particular, for $100\%$ of degree $d$ integer polynomials $|P|\leq H$ with
a positive leading coefficient  one has   $$\sum_{\substack{1\leq m\leq x\\P(m) \ \mathrm{ is } \  E_3 } }\mathcal L_3(P(m))
=\frac{\mathfrak S_{ P}}{6} x (\log H)^2+O\Bigg( \frac{x(\log H)^2}{(\log x)^{1/4}} \Bigg)$$ uniformly for  $x\leq (\log H)^\delta$. Furthermore, $100\%$ of degree $d$ Bouniakowsky polynomials represent more than  $(\log |P|)^\delta$ integers that are $E_3$.
\end{theorem} 

The analogous function  $\mathcal L_4$ for $E_4$ numbers cannot be written as a combination of the $\Lambda_k$ functions, see Lemma \ref{paparity} for an exact statement.
Crucially, $\Lambda_k$ is supported on numbers of type $E_j$ ($j \leq k$), but its average over  $E_j$ for any single $j\leq k$
remains a positive proportion of the overall average.

\begin{notation} Throughout this paper, $\phi$ denotes the Euler totient function, $\mu$ denotes the Möbius function, and $\ast$ represents the Dirichlet convolution. Additionally, we define the shorthand notation $L(m) := \log m$. We let $\omega(m)$ be the number of distinct prime numbers dividing $m$.
\end{notation}

\section{Generalized von Mangoldt functions} \subsection{M\"obius randomness}
\label{ss:mobiusdisjoint} We fix $d,k,m_1\neq m_2\in \mathbb N$
and  study \begin{equation} \label{def:GmGm} \mathcal G_{(m_1,m_2)}(H;d,k):=
\sum_{|P|\leq H} \Lambda_k(P( m_1)) \Lambda_k(P(m_2)),\end{equation} where the sum is over $P\in \mathbb Z[t]$
of degree $d$ with $P>0$ and $P(m_1),P(m_2)>0$. \begin{proposition}\label{prop:casanatense}
Fix $\delta,\epsilon>0$ and $d,k\in \mathbb N$. For all integers  $1\leq m_1\neq m_2 \leq (\log H)^\delta$ we have 
$$\mathcal G_{\b m}(H;d,k)=2^d H^{d+1}  (k (\log H)^{k-1})^2 \frac{|m_1-m_2|}{\phi(|m_1-m_2|)}  
\left(1+ O (    (\log H )^{- 1+ \epsilon} )\right),$$ where the implied constant is independent of $m_i$ and $H$. 
\end{proposition} The proof will be  given in \S \ref{ss:Come non pensi}. For integers $j_1,j_2\in [0,k]$ 
denote  \begin{align*} &\mathcal G_{(m_1,m_2)}(H;d,k)^{(j_1,j_2)}:=\sum_{  |P|\leq H } \prod_{h=1,2} 
(1\ast \mu L^{j_h})(P(m_h)) (\log P(m_1))^{k-j_h}  \\ \ \textrm{ and } \ 
&\mathcal G^\flat_{(m_1,m_2)}(H;d,k)^{(j_1,j_2)}:= \sum_{  |P|\leq H } \prod_{h=1,2} 
(1\ast \mu L^{j_h})(P(m_h))  ,\end{align*} where the sums are over $P\in \mathbb Z[t]$ of degree $d$ with $P>0$ and $P(m_1),P(m_2)>1$. 
\begin{lemma}[Decomposing $\Lambda_k$]\label{lem:decomposition} We have $$\mathcal G_{\b m}(H;d,k)=\sum_{j_1,j_2=1  }^k
{k\choose j_1 } {k\choose j_2 } (-1)^{j_1+j_2} \mathcal G_{\b m}(H;d,k)^{\b j}, $$\end{lemma} \begin{proof} By definition we have
$\Lambda_k(1) =0$. For $m>1$    write  \begin{equation} \label{eq:pergolesi_euridice} \Lambda_k(m)   = \sum_{j=0}^k
{k\choose j } (\log m)^{k-j}(-1)^j (1\ast \mu L^j)(m)\end{equation}  and note that   $(1\ast \mu L^j)(m)=0$ vanishes 
for $j=0$.\end{proof}
We recall the  bound $ 0\leq \Lambda_k(m)\leq L(m)^k$ in \cite[Equation (3.16)]{MR2647984}.
Using this together with \eqref{eq:pergolesi_euridice} and induction  yields \begin{equation}
\label{eq:pergolesi_euridice_dove_sei} (1\ast \mu L^{k})(m) \ll_k L(m)^k 
,\end{equation}with an implied constant independent of $m$. \begin{lemma}[$\log$ stabilisation]
\label{lem:stabilisation} Keep the setting  Proposition \ref{prop:casanatense}. For each $j_1,j_2\in [1,k]$ we have $$
\mathcal G_{\mathbf m}(H;d,k)^{\b j}= (\log H)^{2k-j_1-j_2} \left( 1 +O\left( \frac{\log \log H}{\log H}\right)\right)
\mathcal G^\flat_{\b m}(H;d,k)^{\b j} +O\left(\frac{H^{d+1}}{\log H}\right),$$ where the implied constant is independent of $m_i$ and $H$.
\end{lemma}  \begin{proof}Denote by $E$ the subset of $P$  appearing in  $\mathcal G_{\mathbf m}(H;d,k)^{\mathbf j}$
such that  $\max\{P(m_1),P(m_2)\}$ strictly  exceeds $ H/(\log H)^{2k+1}$. We have  
$$\#E\leq \sum_{0<t\leq H/(\log H)^{2k+1}} \#\{|P|\leq H: (P(m_1)-t)(P(m_2)-t)=0\}
\ll  \frac{H}{(\log H)^{2k+1}}H^d,$$ since the condition $(P(m_1)-t)(P(m_2)-t)=0$ fixes one coefficient of $P$.
The conditions of Proposition \ref{prop:casanatense} ensure that $\log P(m_i)\ll \log H$, thus,  
by \eqref{eq:pergolesi_euridice_dove_sei} the contribution of $E$ towards  $\mathcal G_{\mathbf m }(H;d,k)^{\b j}$ is 
$ \ll  \#E \cdot (\log H)^{2k} \ll  H^{d+1}/\log H$.

When $P\notin E$ we get  $\max\{P(m_1), P(m_2)\} > H^{d+1}/(\log H)^{2k+1}$, hence,  $ \log (P(m_i)/H) \ll \log \log H$. In particular, 
for each $i,j $ we obtain    $ (\log P(m_i))^j =(\log H)^j +O((\log H)^{j-1}(\log \log H))$, thus,
$$ \mathcal G_{\mathbf m}(H;d,k)^{\mathbf j}= O\left(\frac{H^{d+1}}{\log H}\right)+
(\log H)^{2k-j_1-j_2} \left( 1 +O\left( \frac{\log \log H}{\log H}\right)\right)
\sum_{  \substack{ |P|\leq H  } } \prod_{h=1,2}  (1\ast \mu L^{j_h})(P(m_h)) 
,$$ where the sum over $P$ in the right-hand side  is subject to    $\max\{P(m_i):i=1,2\}> 
H/(\log H)^{2k+1}$ and the conditions present  in $\mathcal G^\flat_{\mathbf m}(H;d,k)^{\mathbf j}$.
Ignoring the condition  involving  $\max\{P(m_i):i=1,2\}$  introduces an error term  of size 
 $ \ll  (\log H)^{2k-j_1-j_2} \#E \cdot (\log H)^{j_1+j_2} \ll H^{d+1}/\log H $ by \eqref{eq:pergolesi_euridice_dove_sei}. \end{proof}
Fix $j\geq 1 $, $z\geq 2 $ and for $m\in \mathbb N$  define  $$\c E_{z,j}(r):=\sum_{\substack{c\in \mathbb N, c\mid m \\ c>z } } 
\mu(c) (\log c)^j.$$ This generalizes the function  $\mathcal E_{z}=\mathcal E_{z,1}$ in \cite[page 691]{MR4542704}.
\begin{lemma}[Exponential sum input] \label{lem:davenport}  Fix $A > 0$. Then for all $j\geq 1$ and $ 1\leq z \leq  y $ we have 
\[\sup_{\alpha \in \mathbb{R}} \left| \sum_{m \in \mathbb{N} \cap [1, y]} \mathscr{E}_{z,j}(m) \mathrm{e}^{i r\alpha} \right| 
\ll  \frac{ y \log y}{(\log z)^{A}},\] where the implied constant depends only on $A$ and $j$. \end{lemma}\begin{proof} 
We write the sum as  $$ \sum_{\substack{ c t \leq y \\ c>z  }} \mu(c) (\log c)^j  \mathrm{e}^{i ct \alpha} = 
\sum_{  t \leq y/z } \sum_{ z< c \leq y/t   } \mu(c) (\log c)^j  \mathrm{e}^{i ct \alpha}  .$$ 
By the theorem of Davenport \cite{davenport} and partial summation, the inner sum is $O((y/t)(\log (y/t))^{-A} ))$,
hence, we obtain the overall bound  \[  \ll  \sum_{  t \leq y/z } \frac{y}{t (\log (y/t))^{A}} 
\leq   \frac{y }{(\log z)^{A}} \sum_{  t \leq y  } \frac{1}{t}\ll  \frac{y \log y}{(\log z)^{A}}
.\qedhere\] \end{proof} 
The following lemma is an  analogue of \cite[Proposition 3.8]{MR4542704}; it can be proved by injecting  Lemma \ref{lem:davenport} into
\cite[Lemmas 3.4-3.5]{MR4542704}. Denote by  $$ \mathcal G^{\flat\flat}_{(m_1,m_2)}(H;d,k)^{(j_1,j_2)}$$ the expression 
defined by   $ \mathcal G^\flat_{(m_1,m_2)}(H;d,k)^{(j_1,j_2)}$ when  $1\ast  \mu L^{j_h}$ is replaced by 
$1\ast ( \mu L^{j_h}\mathds 1_{[1,z]})$.

\begin{lemma}[M\"obius randomness]\label{lemma:eurovision} Keep the setting 
Proposition \ref{prop:casanatense} and fix $j_1,j_2\in [1,k],A>0$ and $\eta\in (0,1)$. Then for all $z, H \geq 1$ such that 
 $\exp((\log H)^\eta) \leq z \leq  H$  we have  \[\mathcal G^\flat_{\b m}(H;d,k)^{\b j}=
\mathcal G^{\flat\flat}_{\b m}(H;d,k)^{\b j}+ O \left( \frac{H^{d+1}}{(\log H)^A} \right),
\] where  the implied constant does not depend  on $m_1, m_2, H$ and $z$. \end{lemma} 
The only   difference compared to  \cite[Proposition 3.8]{MR4542704} is that the condition $z\geq H^{\delta_1}$
is relaxed to $\exp((\log H)^\eta) \leq z $, something that was also   possible 
in  \cite{MR4542704} but   not recorded. This alternative choice of $z$  can be explained by  noting that 
Lemma~\ref{lem:davenport} provides  the succeeding error term upon taking   $y=H (\log H)^{\delta}$ and 
$z=\exp((\log H)^\eta)$: $$  \frac{y \log y }{ (\log z)^{(A+1+\delta)/\eta} }
=  \frac{H  \log (H (\log H)^\delta) } {( (\log H)^{\eta})^{(A+1+\delta)/\eta} }
\ll  \frac{H  \log H }{ (\log H)^A }. $$ This different choice for $z$  allows us to  give a weak
upper bound for $\mathcal G_{\b m}(H;d,k)^{\b j}$ but one that  is sufficient for our purposes.
\begin{lemma} \label{lemma:eurovision234} Keep the setting Lemma \ref{lemma:eurovision}and take  $z=\exp((\log H)^\eta)$. Then 
we have  \[ \mathcal G^{\flat\flat}_{\b m}(H;d,k)^{\mathbf j} \ll H^{d+1}  (\log H )^{\eta( j_1+j_2+2)}(\log \log H),\]where the implied constant is independent  from  $m_i$ and $H$. \end{lemma}  \begin{proof} We can upper-bound the sum by  $$\sum_{c_1,c_2 \leq z}
|\mu(c_1)\mu(c_2)| (\log c_1)^{j_1}(\log c_2)^{j_2} \#\{  P\in \mathbb Z[t]:  \deg(P)=d, |P|\leq H,   c_i \mid P(m_i), i=1,2  \}
.$$ Our choice of $z$ makes apparent that   $ (\log c_i)^{j_i} \leq(\log z)^{j_i} =(\log H)^{\eta j_i}$. The cardinality of $P$ was studied in the arguments  appearing in \cite[Equation (3.4)]{MR4542704}. Writing $l_0= \gcd(c_0,c_1)$ and $l_i=c_i/l_0 $ for $i=1,2$  and using the estimates proved in  \cite[pages 697-698]{MR4542704}  we obtain the  bound$$\#\{  P\in \mathbb Z[t]:  \deg(P)=d,
|P|\leq H,   c_i \mid P(m_i), i=1,2  \} \ll  H^{d+1}  \frac{\gcd(l_0,l_1)}{l_0^2 l_1 l_2 }.$$ This gives the overall estimate  
$$\ll H^{d+1} (\log H)^{\eta (j_1+j_2)} \sum_{\substack{ l_0,l_1,l_2 \in \mathbb N \\  l_1,l_2 \leq z}}
|\mu(l_0 )|  \frac{\gcd(l_0,m_1-m_2)  }{l_0^2 l_1 l_2 }   \ll H^{d+1}  (\log H)^{\eta (j_1+j_2)} \prod_{p\mid m_1-m_2}
\left(1+\frac{1}{p}\right) .$$ Since $|m_1-m_2| \leq (\log H)^\delta$ we may use Mertens' theorem to see that\[ \prod_{p\mid m_1-m_2}
\left(1+\frac{1}{p}\right) \ll \log |m_1-m_2| \ll \log \log H.\qedhere\] \end{proof}
 \begin{lemma}\label{lemma:eurovisionUK000} Fix $\epsilon>0$ and $j_1,j_2 \in [1,k]$. In  the setting of 
 Proposition \ref{prop:casanatense}  we have  $$\mathcal G_{\mathbf m}(H;d,k) =k^2 (\log H)^{2k-2}
\mathcal G^\flat_{\mathbf m}(H;d,k)^{(1,1)} \left( 1 +O\left( \frac{\log \log H}{\log H}\right)\right)
+ O(   H^{d+1}  (\log H )^{2k- 3+\epsilon}) ,$$ where the implied constant is independent from  $m_i$ and $H$.
\end{lemma}  \begin{proof} Injecting Lemma~\ref{lemma:eurovision234}  with $\eta=\epsilon/(4k+4)$
into Lemma~\ref{lemma:eurovision} shows that  $$ \mathcal G^\flat_{\mathbf m}(H;d,k)^{\mathbf j} \ll
 H^{d+1}  (\log H )^{\epsilon/2} (\log \log H)\ll   H^{d+1}  (\log H )^\epsilon.$$ Feeding this into 
 Lemma~\ref{lem:stabilisation} we then obtain $$ \mathcal G_{\b m}(H;d,k)^{\mathbf j} \ll H^{d+1} 
(\log H )^{2k- j_1-j_2+\epsilon} .$$ We may then use this for all pairs $(j_1,j_2)\neq (1,1)$ in Lemma~\ref{lem:decomposition} 
to obtain $$\mathcal G_{(m_1,m_2)}(H;d,k)= k^2 \mathcal G_{(m_1,m_2)}(H;d,k)^{(1,1)}+ O\left(   H^{d+1} 
(\log H )^{2k- 3+\epsilon}  \right). $$ Finally,  alluding to Lemma~\ref{lem:stabilisation}
with $j_1=j_2=1$ shows that  $\mathcal G_{(m_1,m_2)}(H;d,k)$ equals  \[k^2 (\log H)^{2k-2}
\mathcal G^\flat_{\mathbf m}(H;d,k)^{(1,1)} \left( 1 +O\left( \frac{\log \log H}{\log H}\right)\right)
+ O(   H^{d+1}  (\log H )^{2k- 3+\epsilon}   ),\] which concludes the proof. \end{proof} 
\subsection{Proof of Proposition \ref{prop:casanatense}} \label{ss:Come non pensi} The main term in 
Lemma \ref{lemma:eurovisionUK000} contains  $\mathcal G^\flat_{\mathbf m}(H;d,k)^{(1,1)}$, which coincides with the 
 sum in  \cite[Definition 3.1]{MR4542704}. Hence, by \cite[Theorem 3.1]{MR4542704} one has  $$ \mathcal G^\flat_{\mathbf m}(H;d,k)^{(1,1)}= 2^d H^{d+1} \frac{|m_1-m_2|}{\phi(|m_1-m_2|)} +O\left( \frac{ H^{d+1}}{ \log H}\right).$$ Combining this 
with Lemma \ref{lemma:eurovisionUK000} we obtain  $$ \mathcal G_{\mathbf m}(H;d,k) =2^d   k^2  H^{d+1}  (\log H)^{2k-2}
\frac{|m_1-m_2|}{\phi(|m_1-m_2|)}  + O (   H^{d+1} (\log H )^{2k- 3+3\epsilon} ),$$  where we used the bound 
\[\frac{|m_1-m_2|}{\phi(|m_1-m_2|)}  \ll \prod_{p\leq |m_1-m_2|} (1-p^{-1})^{-1} \ll \log |m_1-m_2| \ll \log x \ll \log \log H   \] by Mertens' theorem. This completes   the proof of Proposition \ref{prop:casanatense}.
\subsection{Proof of Theorem \ref{thm:Lambda_k}} \label{proofmaintheorem} Opening up the second moment in the  left-hand side of 
Theorem \ref{thm:Lambda_k} we get \begin{equation}\label{eq:karotsaki} V_1-2k_1\cdots k_n x (\log H)^{k-n}    
V_2 +(k_1\cdots k_n x (\log H)^{k-n})^2 V_3 ,  \end{equation} where \begin{align*}
V_1& := \sum_{1\leq m_1,m_2\leq x } \prod_{i=1}^n  \sum_{\substack{P \in \mathbb Z[t], \deg(P)=d_i,P>0
 \\|P|\leq H, P(m_1),P(m_2)>0 }}  \Lambda_{k_i}(P(m_1)) \Lambda_{k_i}(P(m_2)), \\
V_2& :=  \sum_{1\leq m\leq x } \sum_{\substack{\mathbf P \in  \textrm{Poly}(H)
\\ P_i(m)>0 \forall i  }}  \mathfrak S_{\mathbf P }(x) \prod_{i=1}^n \Lambda_{k_i}(P_i(m)) , \\
V_3& := \sum_{\mathbf P \in  \textrm{Poly}(H)} \mathfrak S_{\mathbf P }(x)^2.\end{align*} 
Let us   recall that \cite[Lemma 4.7]{MR4542704} states that  for  $ x \leqslant H^{1/4}$ we have
\begin{equation}\label{lem:ch'io muti consiglio", Guglielmo}V_3 = 2^d H^{d+n}  \prod_{\substack{\mathrm{prime } \ \ell 
\\  \ell \leq \log x}}  \gamma_n(\ell) + O(H^{d+n-1/2}),\end{equation} where $$
\gamma_n(\ell) := 1 - \frac{1}{\ell} + \frac{\ell^{n-1}}{(\ell - 1)^n}=1+\frac{n}{\ell^2}
+O_n\left( \frac{1}{\ell^3}  \right) $$  and the implied constant  is independent  of $x$ and $H$. 

We split $V_1$ as  $V_1= V_1^{\textrm{diag}}+V_1^{\textrm{non-diag}},$ where $V_1^{\textrm{diag}} $ denotes the contribution of the terms $m_1=m_2$ towards $V_1$. Recalling  \eqref{def:GmGm} we apply Proposition \ref{prop:casanatense} to write  
$V_1^{\textrm{non-diag}}$ as \begin{align*} &\sum_{1\leq m_1\neq m_2\leq x }  \prod_{i=1}^n  \mathcal G_{\mathbf m}(H;d_i,k_i)\\
 &=\left(1+ O (    (\log H )^{- 1+ \epsilon} )\right)  2^d H^{d+n}  (\log H)^{2k-2n} k_1^2\cdots k_n^2
\sum_{1\leq m_1\neq m_2\leq x }  \Big( \frac{|m_1-m_2|}{\phi(|m_1-m_2|)}  \Big)^n .\end{align*}   The sum over $m_1 \neq m_2$ in 
the last display  coincides with  the quantity    $2T_0(x)$ (when $M=1$) in the proof  of \cite[Lemma 4.4]{MR4542704}. In that lemma it was proved that  $$\sum_{1\leq m_1< m_2\leq x }  \Big( \frac{|m_1-m_2|}{\phi(|m_1-m_2|)}  \Big)^n= \frac{x^2}{2}
\prod_{\substack{\mathrm{prime } \ \ell  \\  \ell =2 }}^\infty  \gamma_n(\ell) +O(x^{3/2}).$$ Truncating the product in the 
right-hand side to primes  $\ell \leq \log x $ in order to get something closer to \eqref{lem:ch'io muti consiglio", Guglielmo} we introduce  an error term of size $O(x^2/\log x)$ due to $\gamma_n(\ell)=1+n \ell^{-2} +O_n(\ell^{-3})$. Thus, 
\begin{equation}\label{Dunque si sfoga}V_1^{\textrm{non-diag}} =  2^d H^{d+n}  (\log H)^{2k-2n} k_1^2\cdots k_n^2 x^2
\prod_{\substack{\mathrm{prime } \ \ell \\  \ell \leq  \log x}} \gamma_n(\ell)\left(1+ O (    (\log x)^{- 1 } )\right).\end{equation}
Before proceeding to the estimation of $V_1^{\textrm{diag}}$ let us note the following: \begin{remark}\label{Vivrai ma sempre in guerra}
One  can use the bound $\Lambda_k\leq L^k$ to give a crude bound for $V_1^{\textrm{diag}}$, however, this would eventually result in a weaker version of Theorem \ref{thm:Lambda_k}, in which the asymptotic holds in a smaller range for $x$. In light of \begin{equation}\label{Lo so ma il mio timore} M (\log M)^{2k-1} \ll \sum_{m\leq M } \Lambda_k(m) ^2\ll M
(\log M)^{2k-1}\end{equation} we expect that \begin{equation}\label{Misero! In quale abisso}
V_1^{\textrm{diag}}\ll x H^{d+n} (\log H)^{2k-n}.\end{equation} To justify \eqref{Lo so ma il mio timore} we use 
$\Lambda_{k+1}=\Lambda_k L+ \Lambda_k \ast \Lambda$ from \cite[Equation (3.15)]{MR2647984} in order to get 
$\Lambda_{k+1}\geq \Lambda_k L$ and by induction,  $\Lambda_{k}\geq \Lambda_1 L^k=\Lambda L^{k-1}$. Hence,
$$ \sum_{m\leq M } \Lambda_k(m) ^2 \gg (\log M)^{2k-2}  \sum_{m \in (M/2,M] }   \Lambda(m) ^2  \gg M (\log M)^{2k-1}
 $$   by restricting to primes $m$ and using  the Prime Number Theorem.  For the upper bound   we use   $  \Lambda_k \leq L^k$
  from \cite[Equation (3.16)]{MR2647984}   to get   $$ \sum_{m\leq M } \Lambda_k(m) ^2\leq  (\log M)^k\sum_{m\leq M } \Lambda_k(m) 
\ll M (\log M)^{2k-1}$$ by \cite[Equation (3.21)]{MR2647984}.\end{remark} 
Let us now continue by verifying \eqref{Misero! In quale abisso} when $x\geq (\log H)^n$. By $\Lambda_k \leq L^k$ we obtain  
\begin{equation} \label{Rosique}  V_1^{\textrm{diag}} =  \sum_{1\leq m\leq x }  \sum_{\substack{\mathbf P \in  \textrm{Poly}(H)\\P_i(m)>0\forall i }}  \prod_{i=1}^n \Lambda_{k_i}(P_i(m))^2 \ll   (\log H)^{k}  \sum_{\mathbf P \in  \textrm{Poly}(H)} 
\psi_{\mathbf P}(x)=: (\log H)^{k}  Z, \ \textrm{ say},\end{equation} where  $$\psi_{\mathbf P}(x):=
 \sum_{\substack{m\leq x\\P_i(m)>0\forall i }}   \prod_{i=1}^n \Lambda_{k_i}(P_i(m)).$$ It remains to prove 
 $Z\ll H^{d+n} x(\log H)^{k-n}$. By Cauchy's inequality we obtain  \begin{align*} Z^2   &\ll  H^{d+n} 
 \sum_{\mathbf P \in  \textrm{Poly}(H)}  \psi_{\mathbf P }(x)^2 =H^{d+n} (V_1^{\textrm{diag}}+V_1^{\textrm{non-diag}})\\
&\ll   H^{d+n}  ( (\log H)^k Z + H^{d+n}  (\log H)^{2k-2n}  x^2) \end{align*} by \eqref{Dunque si sfoga} and \eqref{Rosique}.
 If the  right hand side  is dominated   by the second term  then   $$ Z^2 \ll  (H^{d+n}  (\log H)^{k-n}  x)^2,$$ which is acceptable. 
If the first term dominates then  by $x\geq (\log H)^n$ we obtain   $$ Z \ll  H^{d+n}   (\log H)^{k}  = H^{d+n}   (\log H)^{k- n}
(\log H)^n \leq H^{d+n}   (\log H)^{k- n} x,$$ which is satisfactory. Bringing together   \eqref{Dunque si sfoga} and \eqref{Misero! In quale abisso} shows that when $x\geq (\log H)^n$ then   \begin{equation}\label{finalv1} V_1= 2^d H^{d+n}  (\log H)^{2k-2n} k_1^2\cdots k_n^2 x^2 \prod_{ \ell \leq \log x}  \gamma_n(\ell)  +O\Big( \frac{ H^{d+n} x^2}{ (\log H)^{2n-2k} } \Big( \frac{(\log H)^n}{x}
+ \frac{1}{\log x} \Big) \Big).\end{equation} To estimate the mixed term $V_2$, we apply  Knafo's generalization of the Siegel--Walfisz theorem for  $\Lambda_k$ \cite[Theorem 5]{knafo}, specializing it to the  range $T =\exp(\sqrt{ \log x})$ and $q \leq (\log y)^A$. Following  \cite[Definition 1, p. 332]{knafo}, there exist explicit constants $a_n(q)$ that satisfy $a_0(q)=1$ and 
the bound $a_n(q)\ll_n (\log \log q)^n $ by  \cite[Lemma 2]{knafo}, yielding the following result:
\begin{lemma}[Knafo]\label{lemma:knafo} Fix $k\geq 1 $ and $A,B>0$. Then for all $y,q\geq 1$ with 
$q\leq (\log y)^A$ and all $b\in (\mathbb Z/q\mathbb Z)^*$  we have  $$ 
\sum_{\substack{1\leq m \leq y  \\ m\equiv b \md q  }} \Lambda_k(m)
= \frac{ y }{\phi(q)}\Big[\sum_{j=1}^k  {k \choose j } a_{j-1}(q)(\log y)^{k-j}\Big]+
O\Big(\frac{y}{(\log y)^B}    \Big) ,$$ where the implied constant depends only on $A,B$ and $k$.
\end{lemma}  \begin{corollary}\label{cowboys from hell:pantera}Fix $k\geq 1 $ and $A,B,C>0$. Then for all $y,q\geq 1$ with 
$q\leq (\log y)^A$, all $b\in (\mathbb Z/q\mathbb Z)^*$ and all  $y (\log y)^{-C}\leq Y \leq y (\log y)^C$ we have 
$$ \sum_{\substack{m \in [Y,Y+y] \\ q \mid m-b  }} \Lambda_k(m)= \frac{   ky }{\phi(q)} (\log y)^{k-1} +
O\Big( \frac{ y(\log y)^{k-7/4}}{\phi(q)}\Big) = \frac{ ky (\log y)^{k-1}  }{\phi(q)}  
(1 + O(   (\log y)^{-3/4} ) ),$$ where the implied constant depends only on $A,B,C$ and $k$.\end{corollary}
\begin{proof}By Lemma \ref{lemma:knafo} with $B=A+C+1$  we obtain  \begin{equation}\label{sospirat} \frac{1}{\phi(q)}\sum_{j=1}^k 
{k \choose j } a_{j-1}(q)  \int_Y^{y+Y}   ( t (\log t)^{k-j} )' \mathrm d t  + O\Big(\frac{Y}{\phi(q)(\log y)^{1+C}}  
\frac{\phi(q)}{(\log y )^A}  \Big). \end{equation} Since $Y\leq y (\log y)^C$ and $\phi(q) \leq q \leq (\log y )^A$
the error term is satisfactory. For a constant $\omega \geq 0 $ we have  $$\int_Y^{y+Y}  
( t (\log t)^{\omega} )' \mathrm d t  =\int_Y^{y+Y}     (\log t)^{\omega-1} (\omega + \log t) \mathrm d t \ll_\omega y (\log y)^{\omega}
,$$ hence, the contribution of $j\neq 1 $ towards \eqref{sospirat} is $$ \ll_k  \frac{y}{\phi(q)}\sum_{j=2}^k 
a_{j-1}(q)   (\log y)^{k-j} \ll  \frac{y}{\phi(q)} \sum_{j=2}^k   (\log \log q ) ^{j-1}  (\log y)^{k-j} 
\ll \frac{y}{\phi(q)}(\log \log q)^k  (\log y)^{k-2},$$ which is acceptable. Similarly,  $$\int_Y^{y+Y}  
( t (\log t)^{\omega} )' \mathrm d t  =\int_Y^{y+Y}     (\log t)^{\omega-1} (\omega + \log t) \mathrm d t 
=\int_Y^{y+Y}(\log t)^\omega  \mathrm d t+O( y (\log y)^{\omega -1 } )  $$ and using the fact that both $\log Y$ and $\log (y+Y)$ 
are $\log y  + O(\log \log y )$ we find that   $$\int_Y^{y+Y}(\log t)^\omega  \mathrm d t = (\log y )^\omega 
\int_Y^{y+Y}(1+O((\log \log t)/\log t)) \mathrm d t   = y (\log y )^\omega \Big( 1+  O\Big ( \frac{ \log \log y }{\log y } \Big)
\Big ).$$ Applying this with $\omega=k-1$ to the term $j=1$ in \eqref{sospirat} completes the proof.\end{proof}
To bound $V_2$, let $W$ be the product of all primes $\ell \leq \log x$. Arguing as in \cite[Equation (4.8)]{MR4542704}, it follows that 
 $$V_2=\sum_{1\leq m\leq x }  \sum_{\substack{ \mathbf  R \in (\mathbb Z/W)[t]^n\\ \deg(R_i)\leq d_i \forall i  }} 
 \mathfrak S_{\mathbf R }(x)  \prod_{i=1}^n  \sum_{\substack{P \in  \mathbb Z[t],|P|\leq H
\\  P_i \equiv R_i \mod W }}  \Lambda_{k_i}(P(m))  ,$$ where  the sum over $P$ is subject to  
$\deg(P)=d_i, P>0$ and $P(m)>0$. Since both  $\Lambda_{k_i} , \mathfrak S_{\b R }(x)$ are non-negative we may
lower bound $V_2$ by  \begin{equation}\label{lowerbound} V_2\geq \sum_{  m\leq x }
 \sum_{\substack{ \mathbf R \in (\mathbb Z/W)[t]^n\\ \gcd(R_i(m),W)=1,  \deg(R_i)\leq d_i \forall i  }} \mathfrak S_{\mathbf R }(x)
 \prod_{i=1}^n  \sum_{\substack{P \in  \mathbb Z[t],|P|\leq H \\  P_i \equiv R_i \mod W }}  \Lambda_{k_i}(P(m)).\end{equation} 
Letting  $P(t)=\sum_{j=0}^d c_j m^j,  N=P(m)-c_0$ and $t=P(m)$ we can  write the  sum over $P$ as 
\begin{equation}\label{lowerbound23} \sum_{c_1,\ldots c_d} \sum_{\substack{t>0,|t-N|\leq H\\  t \equiv R_i(m) \mod W }} 
\Lambda_{k_i}(t ), \end{equation}where the sum over $c_1,\ldots, c_d$ is subject to $|c_j|\leq H$, $ c_d>0$ and   $c_j$ being congruent modulo $W$ to the $j$-th coefficients of $R_i$. We have   $$ \log W =\sum_{\ell\leq \log x }\log \ell \ll \log x < (\log H)^\delta, $$ hence $W\leq (\log H)^{O(1)}$, which allows us to  employ Corollary \ref{cowboys from hell:pantera} with 
$q=W,Y=N $ and $y=H$ to deal with the range $t \in [N,N+H]$.  Applying the same corollary again to deal with the 
range $t \in [N-H,N]$ yields the overall estimate  $$ \sum_{\substack{t>0,|t-N|\leq H\\  t \equiv R_i(m) \mod W }} 
\Lambda_{k_i}(t )= \frac{2k_i H}{\phi(W)} (\log H)^{k_i-1} (1+O( (\log H)^{-3/4})).$$ Injecting this into  \eqref{lowerbound23} yields 
$$\sum_{c_1,\ldots c_d} \sum_{\substack{t>0,|t-N|\leq H \\  t \equiv R_i(m) \mod W }} 
\Lambda_{k_i}(t ) =2^{d_i} k_i \frac{H^{1+d_i}}{W^{d_i}}  (\log H)^{k_i-1} (1+O( (\log H)^{-3/4})).$$
This estimate can in turn be fed into \eqref{lowerbound}  resulting in  $$
V_2\geq  2^d (k_1\cdots k_n) \frac{H^{d+n}}{W^d \phi(W)^n}  ( \log H)^{k-n} 
(1+O( (\log H)^{-3/4}))\sum_{  m\leq x }  \sum_{\mathbf R \in (\mathbb Z/W)[t]^n  } \mathfrak S_{\mathbf R }(x)
$$ where the sum over $\b R$ is subject to  $\gcd(R_i(m),W)=1$ and  $ \deg(R_i)\leq d_i$ for all $i$. 
By alluding to  \cite[Lemma 2.7]{MR4542704} we can estimate the inner sum over $\mathbf R$  as 
$$\phi(W)^n W^d \prod_{\ell \leq \log x } \gamma_n(\ell),$$ which is independent of $m$.
Summing over $m$ yields  $$ V_2\geq  2^d (k_1\cdots k_n)  x H^{d+n} 
( \log H)^{k-n}   \Big(\prod_{\ell \leq \log x } \gamma_n(\ell) \Big)
(1+O( x^{-1}+(\log H)^{-3/4}))  .$$  Injecting this estimate together with 
\eqref{lem:ch'io muti consiglio", Guglielmo} and \eqref{finalv1} to \eqref{eq:karotsaki} 
completes the proof of Theorem \ref{thm:Lambda_k}.

\section{$E_k$ numbers}

\begin{lemma}
\label{lem:ABC}Fix $d\in \mathbb N$ and $A>0$. 
Then the number of $P\in \Z[t]$ with $\deg(P)=d,|P|\leq H$ such that there exists $m\in [1,x]$ with $|P(m)|\leq H (\log H)^{-A}$ is 
$\ll x H^{d+1} (\log H)^{-A}$, with an absolute constant that is independent of $x$ and $H$.
\end{lemma}\begin{proof}By the union bound we get 
\[ \leq \sum_{m\leq x} \sum_{\substack{ (c_1,\ldots, c_d)
\in \Z^d\cap[-H,H]^d\\ |t| \leq H (\log H)^{-A} }}\#\left\{|c_0|\leq H: \sum_{j=0}^d c_j m^j =t\right\} \leq\sum_{m\leq x} \sum_{\substack{ (c_1,\ldots, c_d)
\in \Z^d\cap[-H,H]^d\\ |t| \leq H (\log H)^{-A} }} 1 ,\] which is 
$\ll xH^{d+1}/(\log H)^A$. 
\end{proof}

\begin{lemma}
\label{lem:ABC}
Fix $d\in \mathbb N$ and $\delta,j,c_1,c_2,c_3>0$ and let  
$\c C_P\in \mathbb C$ be a constant defined for each $P\in \Z[t]$ such that 
$|\c C_P| \leq (\log |P|)^{c_3}$ for all $P$.
Let 
$f:\mathbb N \to \mathbb C$ be such that $|f(m)|
\leq c_1 (\log (2+m))^{c_2}$ for all $m$.
Then for all $x\leq (\log H)^\delta$ 
the quantity 
$$
\sum_{\substack{ P\in \mathbb Z[t],|P|\leq H \\  \deg(P)=d}} 
\left( \sum_{\substack{ m\leq x\\P(m)>0} } f(P(m)) 
(\log P(m))^j - \c C_P (\log H)^j\right)^2
$$equals 
$$  (\log H)^{2j}\Big( 1   +O\Big(\frac{\log \log H}{\log  H}
\Big)\Big)
\sum_{\substack{ P\in \mathbb Z[t],|P|\leq H \\  \deg(P)=d}} 
\left( \sum_{\substack{ m\leq x\\P(m)>0} } f(P(m)) 
  - \c C_P \right)^2 +O(H^{d+1} (\log H)^{-A} )
$$where the implied constant depends at most on $A,d,j$ and $c_i$.
\end{lemma}
\begin{proof} Our assumptions imply that   $$ \Big| \sum_{\substack{ m\leq x\\P(m)>0} } f(P(m)) 
(\log P(m))^j - \c C_P (\log H)^j\Big|\ll 
(x(\log H)^{c_2}+ (\log H)^{c_3}) 
(\log H)^{j}\ll (\log H)^{lA_0},$$  with 
$A_0=j+\max\{ \delta+c_2,c_3\}$. 
Thus,   the contribution of $P$   in Lemma \ref{lem:ABC} is  
$$ \ll xH^{d+1} (\log H)^{-A+2A_0}\ll   H^{d+1} (\log H)^{\delta-A+2A_0}\leq 
 H^{d+1} (\log H)^{-A/2} $$ when  $A>2(\delta+4A_0)$. As in the proof of Lemma \ref{lem:stabilisation} 
 we see that for the remaining $P$ and all $m$ in $[1,x]$ 
   we have  $ (\log P(m))^j =(\log H)^j +O((\log H)^{j-1}(\log \log H))$.   
Hence, up to an admissible error term,
the quantity in the lemma  equals 
$$ (\log H)^{2j} \Big( 1   +O\Big(\frac{\log \log H}{\log  H} \Big)\Big)
\Osum_{\substack{ P\in \mathbb Z[t],|P|\leq H \\  \deg(P)=d}} 
\left( \sum_{\substack{ m\leq x\\P(m)>0} } f(P(m)) 
  - \c C_P \right)^2,$$ where $\Osum$ is the sum over all $P$ not appearing in Lemma \ref{lem:ABC}.
  We can now complete the sum over $P$ by using once again Lemma \ref{lem:ABC}. \end{proof}

\begin{lemma}
\label{Davide Perez - Subvenite sancti dei}
Fix $d,k\in \mathbb N$ and $\delta,j>0$.  
Then for all $x\leq (\log H)^\delta$ 
the quantity 
$$
\sum_{\substack{ P\in \mathbb Z[t],|P|\leq H \\  \deg(P)=d}} 
\left( \sum_{\substack{ m\leq x\\P(m)>0} } \Lambda_k(P(m)) 
(\log P(m))^j - \mathfrak S_P(x) (\log H)^j\right)^2
$$equals 
$$  (\log H)^{2j}\Big( 1   +O\Big(\frac{\log \log H}{\log  H}
\Big)\Big)
\sum_{\substack{ P\in \mathbb Z[t],|P|\leq H \\  \deg(P)=d}} 
\left( \sum_{\substack{ m\leq x\\P(m)>0} } \Lambda_k(P(m)) 
  -\mathfrak S_P(x)  \right)^2 +O(H^{d+1} (\log H)^{-A} )
$$where the implied constant depends only on $A,d$ and $j$.
\end{lemma}
\begin{proof} The proof follows directly from 
Lemma \ref{lem:ABC} by recalling that $0\leq \Lambda_k \leq L^k$
and noting that \[
\mathfrak S_ P(x) \leq \prod_{\ell \leq \log x } (1-1/\ell)^{-1}
\ll \log \log x.\qedhere\]\end{proof}\subsection{Proof of Theorem 
\ref{thm:bounia2}}

\begin{proof}The proof starts from \cite[Equation 3.15]{MR2647984}, which states$$ \Lambda_{k+1}=\Lambda_k L+ \Lambda_k \ast \Lambda.$$
In particular, $\Lambda \ast \Lambda=- \Lambda_{2}+\Lambda L$, thus, we handle $\Lambda L$ with 
Lemma \ref{Davide Perez - Subvenite sancti dei} with $j=k=1$ and $- \Lambda_{2}$ with Theorem \ref{thm:Lambda_k} with $n=1,k_1=2$  
to obtain \begin{equation} \label{Abel17} \sum_{\substack{ P\in \mathbb Z[t],|P|\leq H \\  \deg(P)=d}} \left|
\sum_{\substack{1\leq m\leq x\\P(m)>0 } } (\Lambda \ast \Lambda)(P(m)) - \mathfrak S_{P}(x) x  \log H \right|^2  \ll  
\frac{H^{d+1} (x \log H)^2} {\log x}\end{equation}    for $(\log H)(\log \log H) \leq x \leq (\log H)^\delta$.

For a prime $p$ and a positive integer $\alpha$ the quantity $(\Lambda\ast \Lambda)(p^\alpha)$ vanishes when $\alpha=1$ and is 
$O((\log (p^\alpha))^2)$ otherwise. Thus, the contribution of $m,P$ such that $P(m)$ is a  power of a prime is  $$
\sum_{\substack{ P\in \mathbb Z[t],|P|\leq H \\  \deg(P)=d}}  \left| \sum_{\substack{1\leq m\leq x\\P(m)=p^\alpha } } 
(\Lambda \ast \Lambda)(P(m)) \right|^2  \ll x (\log H)^2 \sum_{m\leq x} 
\sum_{\substack{ P\in \mathbb Z[t],|P|\leq H \\  \deg(P)=d,P(m)=p^\alpha } } (\Lambda \ast \Lambda)(P(m)).$$
This is $ \ll H^{d+1/2+\epsilon} $ for any fixed $\epsilon>0$ since $(\Lambda \ast \Lambda)(P(m))=O(H^\epsilon)$,
$x\leq (\log H)^\delta$ and  \begin{align*} &\#\{ P\in \mathbb Z[t],|P|\leq H:\deg(P)=d,P(m)=p^\alpha, \alpha =2\} \\ 
\leq  &\sum_{2\leq \alpha \ll \log H}  \sum_{p \ll (x^d H)^{1/\alpha}} \#\{ P\in \mathbb Z[t],|P|\leq H:
\deg(P)=d,P(m)=t\}\ll (\log H) (x^d H)^{1/2} H^d \end{align*} uniformly in $m$. Since $\Lambda\ast \Lambda$ is supported 
on integers $m$ with $\omega(m)\leq 2 $   we deduce  $$ \sum_{\substack{ P\in \mathbb Z[t],|P|\leq H \\  \deg(P)=d}} 
\left| \sum_{\substack{  m\leq x, P(m)>0\\ P(m) \textrm{ is } E_2 } }  (\Lambda \ast \Lambda)(P(m)) - \mathfrak S_{P}(x) x  
\log H  \right|^2  \ll  \frac{H^{d+1} (x \log H)^2}  {\log x}$$  when  $x\leq \log H$. To finish the proof of the first claim of 
Theorem \ref{thm:bounia2} we use the fact that  \[  \omega(t)=2 \Rightarrow  (\Lambda\ast  \Lambda)(t)=2 \prod_{ \textrm{ prime } 
\ell  \mid t} \log \ell =2 \c L_2(t) \]for $t=P(m)$. The second claim of Theorem \ref{thm:bounia2} follows from the first and Chebychev's inequality.  The third  claim of Theorem \ref{thm:bounia2} follows from the second and \cite[Lemma 4.11]{MR485740}, which shows that $100\%$ of Bouniakowsky polynomials  $P$ satisfy $\mathfrak S_P(x)\gg  (\log \log x)^{1-\deg(P)}$.
 \end{proof}\subsection{Proof   of Theorem \ref{thm:bounia33}} \begin{proof}The idea is   the same as in  the proof of  Theorem  \ref{thm:bounia2}, the only difference lying in  the convolution identity at the start of the argument.
More specifically, we seek an identity that expresses the triple convolution $\Lambda_3$ as
a linear combination of functions $L^a  \underbrace{\Lambda  \ast \dots \ast \Lambda}_{b \text{ times}}$. We have 
\begin{align*}
\Lambda_3
&= L \Lambda_2 + \Lambda_2\ast \Lambda
\\&=L (L \Lambda+ \Lambda\ast \Lambda) 
+ (L \Lambda+ \Lambda\ast \Lambda)  \ast \Lambda
\\ &=
L^2 \Lambda+ L (\Lambda\ast \Lambda  )
+L \Lambda\ast \Lambda +
\Lambda\ast \Lambda\ast \Lambda.
\end{align*}
Noting that $L (f \ast g )= ((Lf)\ast g)+  (f\ast (Lg))$
for any arithmetic functions $f,g$, we infer that 
$L(\Lambda \ast \Lambda)= 2 ( \Lambda \ast (L\Lambda))$,
hence, 
$$
\Lambda_3
-
L^2 \Lambda 
-\frac{3}{2}L( \Lambda\ast \Lambda )=
\Lambda\ast \Lambda\ast \Lambda.$$
We can then deal with $L^2 \Lambda$
by Lemma \ref{Davide Perez - Subvenite sancti dei}
and Theorem \ref{thm:Lambda_k}
with $n=1=k_1$, while, 
$L ( \Lambda\ast \Lambda )$ is dealt with by 
\eqref{Abel17} and Lemma \ref{Davide Perez - Subvenite sancti dei}. We deduce that $$\sum_{\substack{ P\in \mathbb Z[t],|P|\leq H \\  \deg(P)=d}} 
\left|
\sum_{\substack{1\leq m\leq x\\P(m)>0 } } 
(\Lambda \ast\Lambda \ast \Lambda)(P(m))
- \frac{1}{2}
\mathfrak S_{ P}(x) x  (\log H)^2
\right|^2  \ll  \frac{H^{d+1} (x (\log H)^2)^2} 
{\log x}$$ for $(\log H)(\log \log H) \leq x \leq (\log H)^\delta$. The terms with $\omega(P(m))=1$ can 
be absorbed into the error term as was done in the proof of the proof of  Theorem 
\ref{thm:bounia2}. 

As $\Lambda \ast\Lambda \ast \Lambda$ is supported on integers $m$ with $\omega(m)\leq 3$ it remains
  remains to show that the terms with $\omega(P(m))=2$ contribute to the error term. We note that the triple 
  convolution  vanishes
when evaluated 
at $p^\alpha q^\beta$ for primes $p\neq q $
and strictly positive integers $\alpha,\beta$.
Furthermore, the contribution of
terms with $\omega(P(m))=p^\alpha q^\beta$ and $\alpha
\geq 2,\beta\geq 2 $  
can be proved to be  $O(H^{d+1/2+\epsilon})$
as in the end of the proof of Theorem 
\ref{thm:bounia2}.Since
$(\Lambda \ast \Lambda \ast \Lambda)(p^2 q) =3
(\log p)^2\log q $, 
the contribution of the terms with $P(m)=p^2 q $ 
is 
$$\ll 
\sum_{\substack{ P\in \mathbb Z[t],|P|\leq H \\  \deg(P)=d}} 
\left|
\sum_{\substack{m\leq x\\P(m)\in \c E} } 
(\Lambda \ast \Lambda \ast \Lambda)(P(m))
\right|^2  \ll  x (\log H)^4
\sum_{\substack{ m \leq x \\ 
p \ll (x^dH)^{1/2}
}} (\log p)^2
\sum_{\substack{ P\in \mathbb Z[t],|P|\leq H \\  
\deg(P)=d,p^2 \mid  P(m)\\
P(m)/p^2 \textrm{ prime}} } 1 $$
by the bounds $
\sum_{m\leq x} 
(\Lambda \ast \Lambda \ast \Lambda)(P(m)) \ll x (\log H)^3$
and $\log (P(m)/p^2) \ll \log H$.
In the sum over $P$ in the right-hand side we 
write  $P=\sum_{j=0}^d c_j t^j $ and fix 
all $c_j$ except $c_0$.
As $(x^d H)^{1/2}\ll H^{1/2+\epsilon}$ for all $\epsilon$,
the sum over $c_0$ is trivially bounded by $\ll 1+ H/p^2$, 
hence, the sum over all $p> (\log H)^2$ gives the overall bound
$$  \ll  x^2 (\log H)^4
\sum_{(\log H)^2 < p \ll (x^dH)^{1/2} }(\log p)^2
H^d (1+ H/p^2)\ll 
x^2 H^{d+1}(\log H)^4
,$$ which is satisfactory. In the remaining range $p\leq (\log H)^2$
we fix   all $c_j$ except $c_0$ and we bring in  
the variable   $N=\sum_{j\neq 0} c_j m^j$. Then 
 the sum over $c_0$ becomes
 $$\#\{|c_0| \leq H :p^2 \mid c_0+N, |c_0+N|/p^2 \textrm{ prime}\}
 = \#\left\{m\in \mathbb Z: \Big|m- \frac{N}{p^2}\Big| \leq \frac{H}{p^2}, |m| \textrm{ prime}\right\} \ll \frac{H}{p^2 \log H}
 $$ by the Brun--Titchmarsh theorem.
 Thus, the overall contribution becomes 
$$ \ll x^2 (\log H)^4
\sum_{
p \ll (x^dH)^{1/2}  
 } (\log p)^2 \frac{H^{d+1} }{p^2 \log H}
\ll x^2 H^{d+1} (\log H)^3.$$
The first claim of 
Theorem \ref{thm:bounia33}
can now be proved by applying  
\[  \omega(t)=3 \Rightarrow 
(\Lambda\ast \Lambda\ast  \Lambda)(t)=6 \prod_{ \textrm{ prime } 
\ell  \mid t} \log \ell =6 \c L_2(t)
\]with $t=P(m)$. The second and third claim of
Theorem \ref{thm:bounia33}
can be proved as in the analogous statements in 
Theorem \ref{thm:bounia2}.
\end{proof}

\subsection{$E_k$ for $k\geq 4$.}
Define $\Lambda^{(1)}=\Lambda$ and by induction for all positive integers $a$ let $\Lambda^{(a+1)}= \Lambda \ast \Lambda^{(a)}$.
Our proofs of Theorem \ref{thm:bounia2}  and
Theorem \ref{thm:bounia33}
start by making a passage to the $\Lambda_k$ 
functions by an identity that expresses $\Lambda_k$ as a finite linear 
combination of $\Lambda^{(a)}L^b$ and then applying Theorem \ref{thm:Lambda_k}. In our last 
lemma, we show that this strategy does not work for detecting
$E_k$ numbers 
when $k \geq 4$.
\begin{lemma} \label{paparity}For any integer $k \geq 4$, 
the  function $\Lambda_k$ cannot be expressed as a finite linear combination of the functions $\Lambda^{(a)}L^b$ for integers 
$a,b\geq 0$
using real coefficients.
\end{lemma}
\begin{proof}  If $\Lambda_k$ is a  
linear combination of 
$L^a \Lambda^{(b)}$, then, by 
$\Lambda_k=L \Lambda_{k-1} + \Lambda \ast \Lambda_{k-1}$
so is $\Lambda \ast \Lambda_{k-1}$.
Hence, there are $c_{a,b}\in \mathbb R$ and $N\in \mathbb N$ 
such that 
$$\Lambda \ast \Lambda_{k-1} = \sum_{0\leq a,b \leq N} c_{a,b}
L^a \Lambda^{(b)}.$$ Let $p,q$ be two distinct primes
and estimate both sides at $pq$. We obtain that 
$$ (\log p)(\log q)((\log p)^{k-2} + (\log q)^{k-2})
 =
 (\Lambda \ast \Lambda_{k-1})(pq)$$ equals $$
 \Lambda^{(2)}
\sum_{0\leq a \leq N} c_{a,2} ( (\log p)+(\log q))^a 
=2 (\log p)(\log q)
\sum_{0\leq a \leq N} c_{a,2} ( (\log p)+(\log q))^a$$
because $\Lambda^{(b)}(pq)=0$ when $b\neq 2 $. Letting 
$x=\log p$ and $ y=\log q$ we infer that $$x^{k-2} + y^{k-2}= P(x+y)
$$ where $P$ is a polynomial with coefficients in $\mathbb R$.

If we fix any prime $q$ then we fix $y$
and the above identity can be seen 
as a polynomial identity in $x$ that holds for infinitely many 
different real numbers, hence, it holds for all real numbers $x$.
The change of variables  $z=x+y$ leads to  
$$P(z)=(z-y)^{k-2}+ y^{k-2},$$ which holds for all $z\in \mathbb R$ 
and all $y$ of the form $\log q$, where $a$ is a prime number.
We can now get a contradiction as follows:
If $k$ is even then $P(0)= 2 y^{k-2}$, 
thus, $P(0)$ takes more than one value,
a contradiction. If $k$ is odd   
then 
$$ \lim_{z\to 0} f'(z)=\lim_{z\to 0} 
(k-2)(z-y)^{k-3}
=(k-2)(-y)^{k-3},$$
which depends on $y$ and   therefore assumes more than 
a single value.\end{proof}

\end{document}